\documentclass[12pt]{article}
\usepackage{amsthm}
\usepackage{amssymb}
\usepackage{indentfirst}
\usepackage{graphicx}

\def\iffujitsu{\ifx{34}}

\def\P{{\bf P}}
\def\E{{\bf E}}

\def\IR{{\bf R}}
\def\X{{\cal X}}
\def\Xt{\tilde{X}}
\def\one{{\bf 1}}

\def\ehat{\hat{e}}
\def\ebar{\bar{e}}
\def\Phat{\widehat{P}}
\def\pihat{\widehat{\pi}}
\def\liml{\lim\limits}

\def\ooN{{1 \over N}}
\def\Uniform{\textrm{Uniform}}
\def\argmin{\textrm{argmin}}

\newtheorem{theorem}{Theorem}
\newtheorem{proposition}[theorem]{Proposition}
\newtheorem{lemma}[theorem]{Lemma}
\newtheorem{corollary}[theorem]{Corollary}
\newtheorem{remark}[theorem]{Remark}
\def\eqref#1{(\ref{#1})}

\begin{document}

\baselineskip=18pt
\parskip=8pt

\centerline{\LARGE\bf Jump Markov Chains and}

\bigskip
\centerline{\LARGE\bf Rejection-Free Metropolis Algorithms}

\bigskip
\centerline{by}
\bigskip
\centerline{\large Jeffrey S.\ Rosenthal\footnote{Department of
		Statistical Sciences, University of Toronto, Canada},
	Aki Dote\footnote{Department of Electrical and Computer
			Engineering, University of Toronto, Canada}%
	$^,$\footnote{Fujitsu Laboratories Ltd., Kanagawa, Japan},
	Keivan Dabiri$^2$,}
\medskip
	\centerline{\large Hirotaka Tamura$^3$,
        Sigeng Chen$^1$,
	and Ali Sheikholeslami$^2$}

\bigskip \centerline{(Version of \today)}

\bigskip
\begin{quote}
\baselineskip=12pt
\noindent \bf Abstract. \rm
We consider versions of the Metropolis algorithm which avoid the
inefficiency of rejections.  We first illustrate that a natural Uniform
Selection Algorithm might not converge to the correct distribution.  We
then analyse the use of Markov jump chains which avoid successive
repetitions of the same state.  After exploring the properties of jump
chains, we show how they can exploit parallelism in computer hardware to
produce more efficient samples.  We apply our results to the Metropolis
algorithm, to Parallel Tempering, to a Bayesian model, to a
two-dimensional ferromagnetic 4$\times$4 Ising model, and to
a pseudo-marginal MCMC algorithm.
\end{quote}


\section{Introduction}

The Metropolis algorithm~\cite{metropolis,hastings}
is a method of designing a Markov chain
which converges to a given target
density $\pi$ on a state space $S$.
Such Markov chain Monte Carlo (MCMC) algorithms have become
extremely popular in statistical applications and have led to a tremendous
amount of research activity (see e.g.~\cite{handbook} and the many
references therein).

The Metropolis algorithm produces a Markov chain
$X_0,X_1,X_2,\ldots$ on $S$, as follows.
Given the current state $X_n$, the
Metropolis algorithm first proposes a new state $Y_n$ from a symmetric proposal
distribution $Q(X_n,\cdot)$.  It then accepts the new state (i.e., sets
$X_{n+1}=Y_n$) with
probability $\min\left( 1, \ {\pi(Y_n) \over \pi(X_n)} \right)$, i.e.\
if $U_n < {\pi(Y_n) \over \pi(X_n)}$ where $U_n$ is an
independent Uniform[0,1] random variable.  Otherwise, it rejects the
proposal (i.e., sets $X_{n+1}=X_n$).
This simple algorithm ensures that the Markov chain has
$\pi$ as a stationary distribution.

With this algorithm, the expected value $\E_\pi(h)$ of a function
$h:S\to\IR$ can then be estimated by the usual estimator, $\ehat_K = {1
\over K} \sum_{n=1}^K h(X_n)$.  The Strong Law of Large Numbers (SLLN) for
Markov chains (e.g.~\cite[Theorem~17.0.1]{MT}) says that assuming that
$\E_\pi(h)$ is finite, and that the Markov chain is irreducible with
stationary distribution $\pi$, we must have $\lim_{K\to\infty} \ehat_K =
\E_\pi(h)$, i.e.\ this estimate $\ehat_K$ is {\it consistent}.  For
example, if $h=\one_A$ is the indicator function of an event $A$, then
$\lim_{K\to\infty} \ehat_K = \P(A)$.  Or, if $h=g^k$ is a power of some
other function $g$, then $\lim_{K\to\infty} \ehat_K = \E_\pi(h) =
\E_\pi(g^k)$.  Consistency is thus a useful property which
guarantees asymptotically accurate estimates of any quantity of interest.


One problem with the Metropolis algorithm is that it might reject many
proposals, leading to inefficiencies in its convergence.  Indeed, in
certain contexts the optimal Metropolis algorithm should reject over three
quarters of its proposals~\cite{RGG,statsci}.  Each rejection involves
sampling a proposed state, computing a ratio of target probabilities, and
deciding not to accept the proposal, only to remain at the current state.
These rejections are normally considered to be a necessary evil of the
Metropolis algorithm.  However, recent technological advances have allowed
for exploiting parallelism in computer hardware, computing all potential
acceptance probabilities at once, thus allowing for the possibility of
skipping the rejection steps and instead accepting a move every time.
Such rejection-free algorithms can be very efficient, but they must be
executed correctly or they can lead to biased estimates, as we now explore.


\section{The Uniform Selection Algorithm}

A first try at a rejection-free Metropolis algorithm might be as follows.
Suppose that from a state $x$,
one of a (large, finite) collection of states $y_1,y_2,\ldots,y_k$
(all distinct from $x$)
would have been proposed uniformly at random.  Then, sample
$U \sim \Uniform[0,1]$, and consider the
sub-collection of states $C := \{y_i : U < \pi(y_i)/\pi(x) \}$ that
``would'' have been accepted, and then pick one of the states in $C$
uniformly at random.
(If $C$ happens to be empty, then we immediately
re-sample $U$ and try again.  Technically speaking, that would be a
``rejection'', though its probability is small.)
This algorithm will always move somewhere, so there is no rejection.
However, this algorithm is different from true MCMC,
and might not converge to $\pi$, as we now show.

\medskip\noindent {\bf Example~1: }
Suppose the state space $S=\{1,2,3\}$,
with $\pi(1)=1/2$, $\pi(2)=1/3$, and $\pi(3)=1/6$,
as in Figure~\ref{threeexfig},
and suppose that from each state $x$, the chain proposes to move either to
$x-1$ or to $x+1$ with probability 1/2 each (where proposals to~0 or to~4
are always rejected).
\begin{figure}[ht]
\centerline{ {\includegraphics[width=8cm]{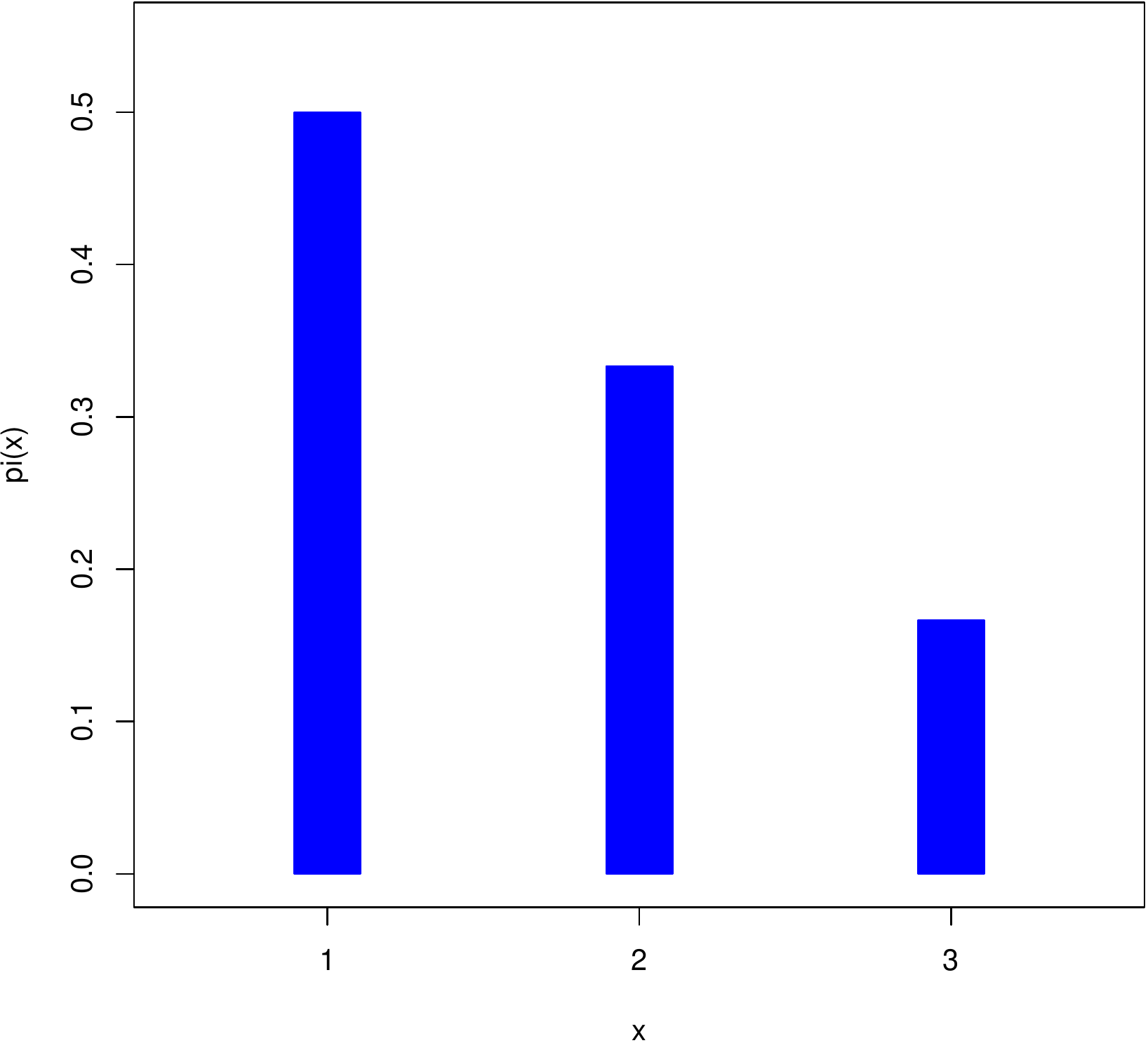}} }
\caption{\bf The target distribution for Example~1.}
\label{threeexfig}
\end{figure}
%
%
In this example,
the Metropolis algorithm would have Markov chain transition
probabilities as in Figure~\ref{metropolisfig},
which are easily computed to
have the correct limiting stationary distribution
$\pi=(1/2, 1/3, 1/6)$ as they must.
\begin{figure}[ht]
\centerline{ \fbox{\includegraphics{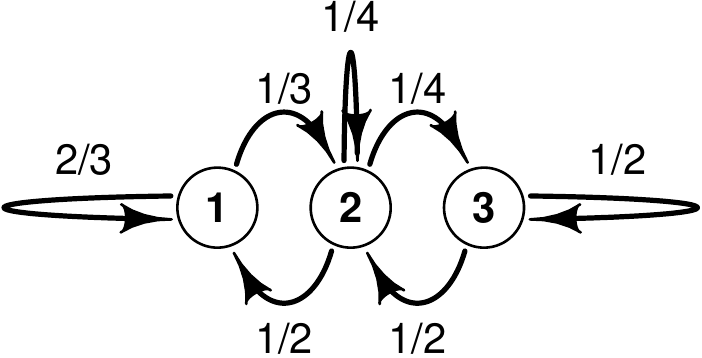}} }
\caption{\bf The Metropolis chain for Example~1.}
\label{metropolisfig}
\end{figure}
However, the Uniform Selection algorithm would have Markov chain transition
probabilities as in Figure~\ref{uniformselfig}, with
limiting stationary distribution easily computed to
instead be $(3/5, 4/15, 2/15)$ which is significantly different.
For example, from state~2, the usual Metropolis algorithm would accept
a proposed move to state~1 with
probability~1, and would accept a proposed move to state~3 with probability
$(1/6) \, / \, (1/3) = 1/2$,
so it would be {\it twice} as likely to move
to state~1 as to move to state~3.
But for the above Uniform Selection version, if $U>1/2$ then the subset $C$
would consist of just the single state~1 so it would always move to
state~1, or if $U<1/2$ then the subset $C$ would consist of the two
states~1 and~3 so it would move to state~1 or state~3 with probability
1/2 each, so overall it would move to state~1 with probability
$(1/2)(1) + (1/2)(1/2) = 3/4$ or to state~3 with probability $(1/2)(0) +
(1/2)(1/2) = 1/4$, i.e.\ it would now be {\it three} times as
likely to move to state~1 as to move to state~3, not twice.
This illustrates that this Uniform Selection algorithm will converge to the
wrong distribution, i.e.\ it will {\it fail} to converge to the
correct target distribution.
\qed


\begin{figure}[ht]
\centerline{ \fbox{\includegraphics{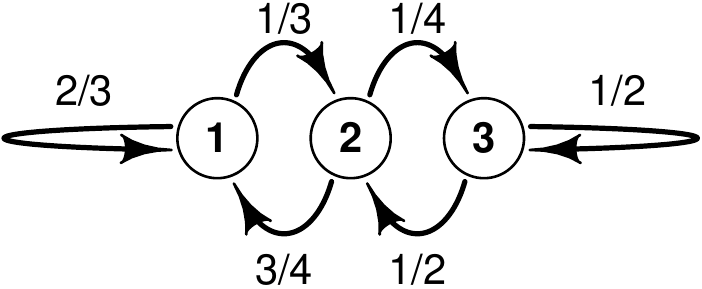}} }
\caption{\bf The Uniform Selection chain for Example~1.}
\label{uniformselfig}
\end{figure}

Our second example shows that Uniform Selection can even cause a Markov
chain to become {\it transient}.

\medskip\noindent {\bf Example~2: }
Suppose now that the
state space is the set $S = \{0,1,2,3,\ldots\}$ of all non-negative
integers, with target distribution $\pi$ defined by writing the
argument $x$ as $x=4a+b$ where $0 \le b \le 3$ is the remainder upon
dividing $x$ by~4, and defining (see Figure~\ref{infexfig})
$$
\pi(x)
\ = \ \pi(4a+b)
\ = \ {1 \over 135} \, (8/9)^a \, 2^b
\, ,
\quad 0 \le b \le 3, \quad a=0,1,2,\ldots
$$

\begin{figure}[ht]
\centerline{ {\includegraphics[width=14cm]{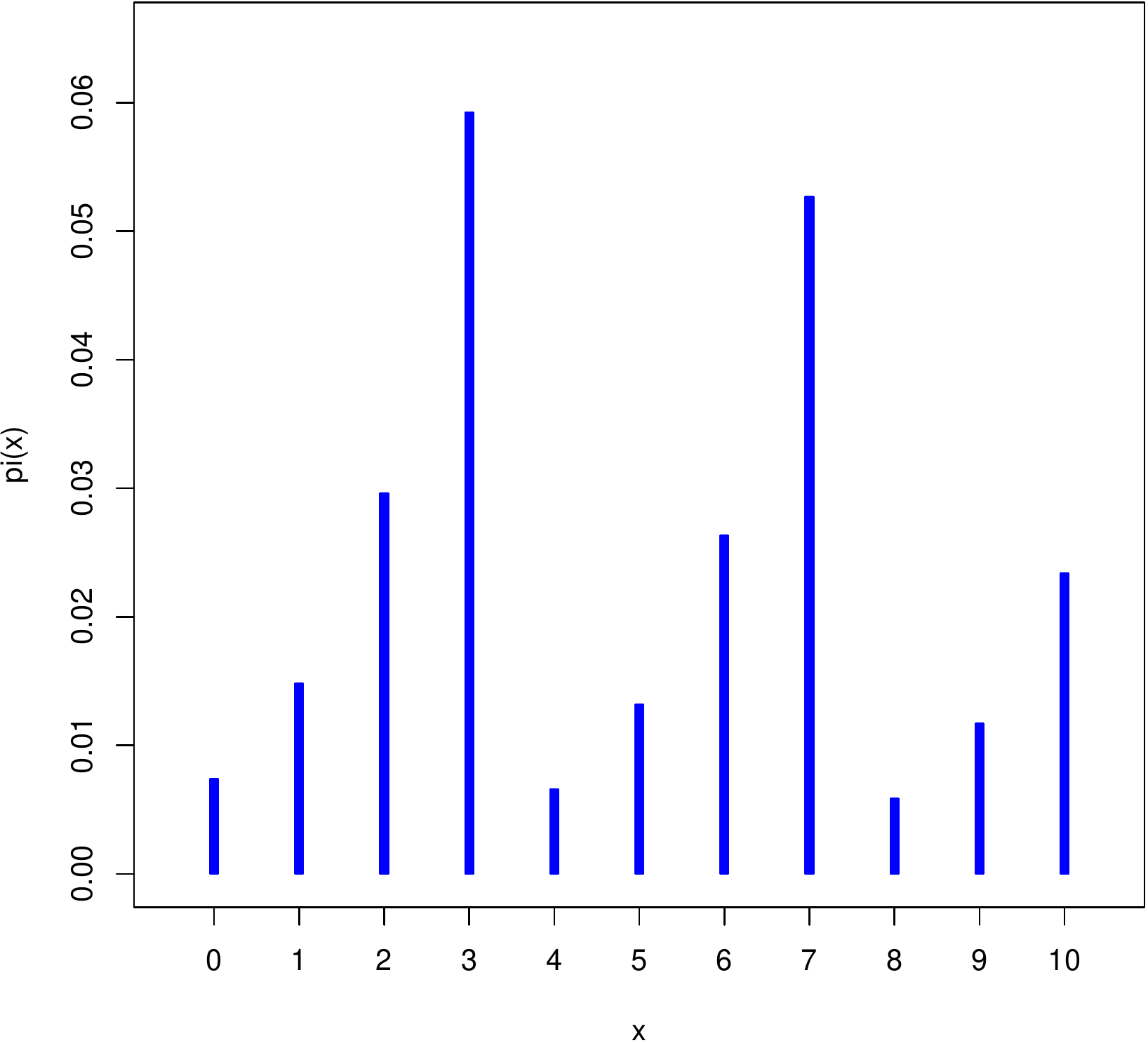}} }
\caption{\bf Part of the target distribution for Example~2.}
\label{infexfig}
\end{figure}

\noindent
As a check,
$$
\sum_{x=0}^\infty \pi(x)
\ = \ \sum_{a=0}^\infty {1 \over 135} \, (8/9)^a \, (2^0+2^1+2^2+2^3)
\ = \ {1 \over 135} \, \left( 1 \over 1-(8/9) \right) \, (15)
\ = \ 1
\, ,
$$
i.e.\ $\pi$ is
a valid probability distribution.
The Metropolis algorithm chain for this example is given by
Figure~\ref{infmetfig}, and it has the correct limiting stationary
distribution $\pi$, as it must.
\begin{figure}[ht]
\centerline{ \fbox{\includegraphics{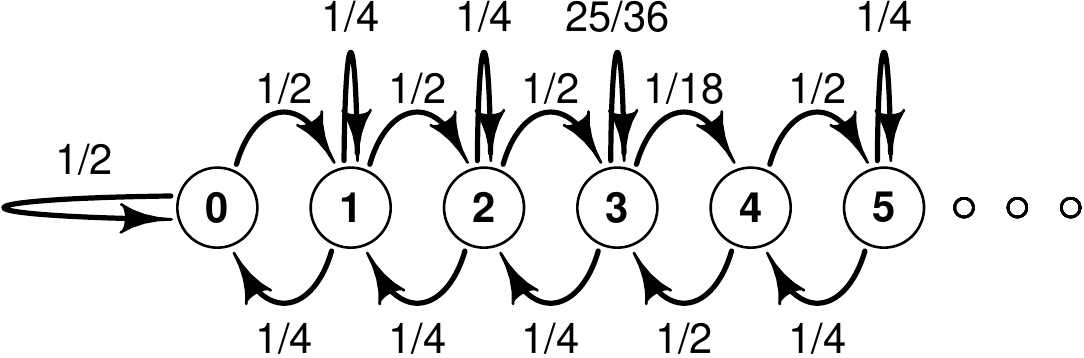}} }
\caption{\bf The Metropolis chain for Example~2.}
\label{infmetfig}
\end{figure}
However, the Uniform Selection chain is instead given by
Figure~\ref{infuniffig}.
We prove in the Appendix that this Uniform Selection
chain is transient, and in fact:

\begin{proposition}\label{appendixprop}
If the Uniform Selection chain for Example~2 begins at state $4a$
for some positive integer $a \ge 2$,
then the probability it will ever reach the state~3 is 
$\le (8/9)^{a-1} < 1$.
\end{proposition}

\noindent
That is, the Uniform Selection
chain might fail to ever reach the optimal value.
For example, if $X_0=100$, then $a=25$ and the probability of failure
is at least $1 - (8/9)^{24} > 0.94 = 94\%$.
This is also illustrated by the simulation\footnote{Performed using
the C program available at: http://probability.ca/rejfree.c}
in Figure~\ref{exsimfig} with initial state $X_0=100$.
\qed



These examples show that the Uniform Selection algorithm may
converge to the wrong limiting distribution, and thus should not be used
for sampling purposes.

Example~2 also has implications for optimisation.
Any Markov chain which gives consistent estimators can be used to find the
mode (maximum value) of $\pi$, either
by running the chain for a long time and taking its empirical sample mode,
or by keeping track of the largest value $\pi(x)$ over all samples visited.
However, Example~2 shows that a Uniform Selection chain could be transient
and thus fail to find or converge to the maximum value at all.
Of course, if the state space $S$ is required to be {\it finite}, then any
irreducible chain will eventually find the optimal value.  However,
the time to find it could be extremely large.  Indeed, the
Appendix also shows that if Example~2
is instead truncated at a large value $4L$, then
each attempt from $4L$ to reach state~3 before returning to $4L$
would have probability less than $(8/9)^{L-1}$ of success.  Hence,
the expected time to ever reach the state~3 would be exponentially large
as a function of~$L$, and the chain
would still spend nearly all of its time very near to the state~$4L$, so
its samples and sample mean and sample mode would all be extremely far
from the true optimal state~3.

\begin{figure}[ht]
\centerline{ \fbox{\includegraphics{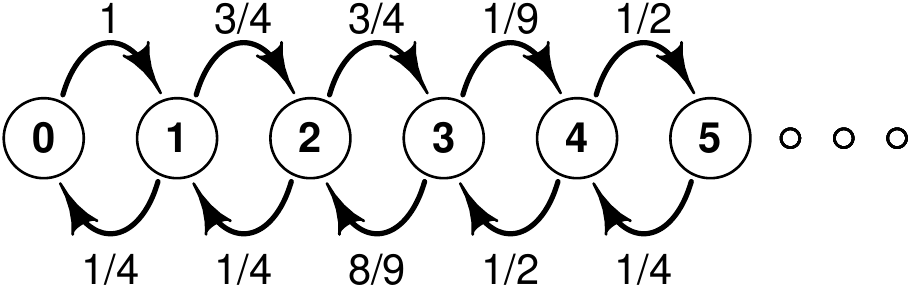}} }
\caption{\bf The Uniform Selection chain for Example~2.}
\label{infuniffig}
\end{figure}

\begin{figure}[ht]
\centerline{ {\includegraphics[width=12cm]{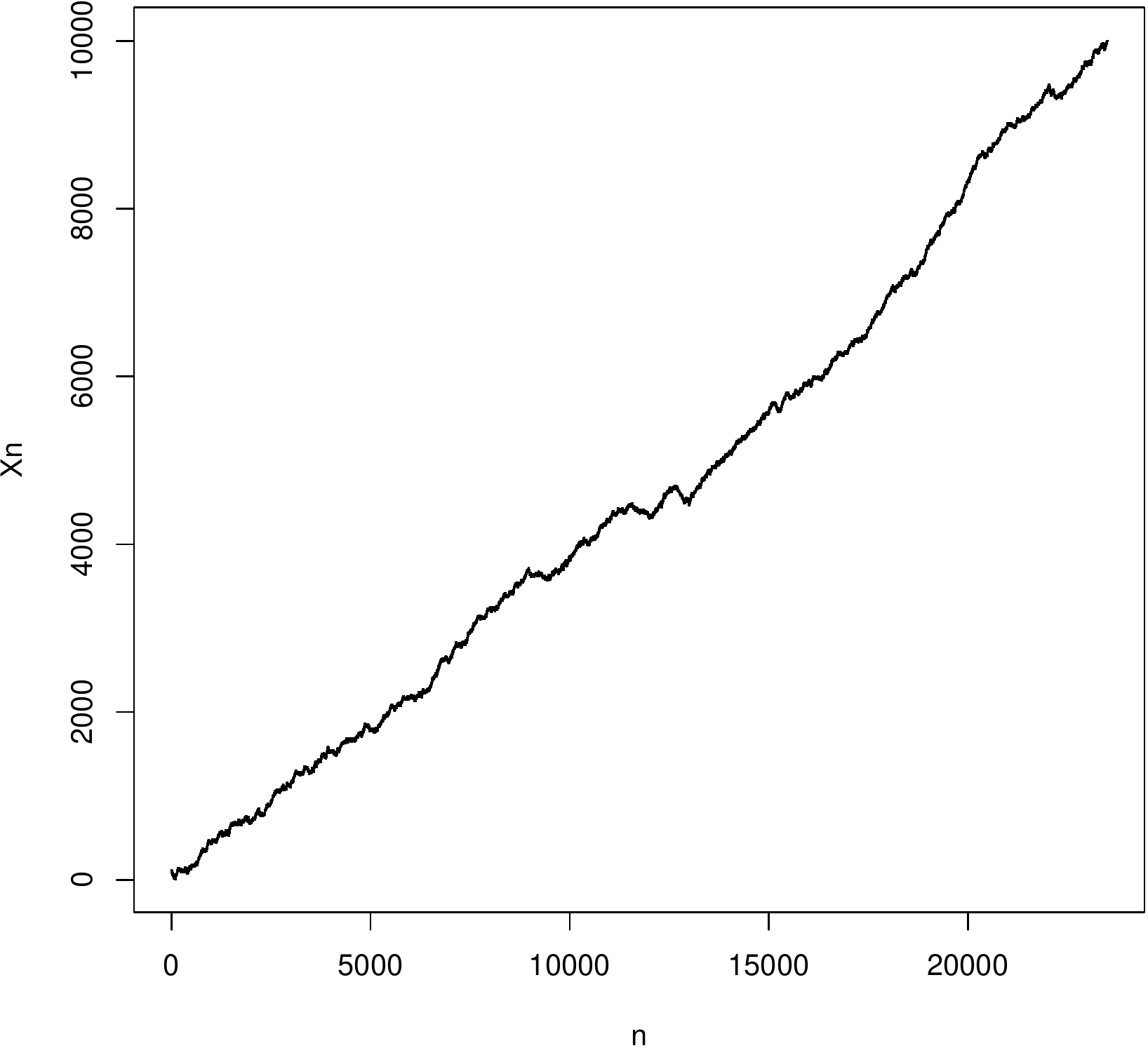}} }
\caption{\bf Output from the Uniform Selection chain for Example~2.}
\label{exsimfig}
\end{figure}

\section{The Jump Chain}

Due to the problems with the Uniform Selection Algorithm identified above,
we instead turn attention to a more promising avenue, the Jump Chain.
Our definitions are as follows.

Let $\{X_n\}$ be an irreducible Markov chain on a state space $S$
(the ``original chain'').  For ease of exposition we initially
assume that $S$ is
finite or countable, though we later (Theorem~\ref{contthm})
extend this to general Markov chains with densities.
To avoid trivialities, we assume throughout that $|S|>1$.

Given a run $\{X_n\}$ of the Markov chain, we define the {\it Jump
Chain} $\{J_k\}$ to be the same chain except omitting any immediately repeated
states, and the {\it Multiplicity List} $\{M_k\}$ to count the number of
times the original chain remains at the same state.
For example, if the original chain $\{X_n\}$ began
$$
\{X_n\} = ( a, b, b, b, a, a, c, c, c, c, d, d, a, \ldots )
\, ,
$$
then the jump chain $\{J_k\}$ would begin
$$
\{J_k\} = ( a, b, a, c, d, a, \ldots )
\, ,
$$
and the corresponding multiplicity list $\{M_k\}$ would begin
$$
\{M_k\} = ( 1, 3, 2, 4, 2, \ldots )
\, .
$$
The concept of jump chains arises frequently for Markov processes,
especially for continuous-time processes where they are often defined
in terms of infinitesimal generators; see e.g.\
Section~4.4 of \cite{durrett} or
Proposition 4.4.20 of \cite{spbook}.
Here we develop the essential properties that we will use below.
Most of these properties are already known in the context of
(reversible) Metropolis-Hastings algorithms; see
Remark~\ref{jumprefremark} below.

To continue, let
$$
P(y|x)
\ = \
\P[X_{n+1} = y \, | \, X_n=x]
\, ,
\quad x,y\in S
$$
be the transition probabilities for the original chain $\{X_n\}$.  And, let
\begin{equation}\label{alpha}
\alpha(x)
\ = \ \P[X_{n+1} \not= x \, | \, X_n=x]
\ = \ \sum_{y \not= x} P(y|x)
\ = \ 1 - P(x|x)
\end{equation}
be the ``escape''
probability that the original chain will move away from $x$ on the next step.
Note that since the chain is irreducible and $|S|>1$, we must have
$\alpha(x)>0$ for all $x\in S$.
We then verify the following properties of the jump chain.

\begin{proposition}
The jump chain $\{J_k\}$ is itself a Markov chain, with
transition probabilities $\Phat(y|x)$
specified by $\Phat(x|x)=0$, and for $y\not=x$,
\begin{equation}\label{jumptrans}
\Phat(y|x)
\ := \ \P[J_{k+1} = y \, | \, J_k=x]
\ = \ { P(y|x) \over \sum_{z \not= x} P(z|x) }
\ = \ { P(y|x) \over \alpha(x) }
\, .
\end{equation}
\end{proposition}

\begin{proof}
It follows from the definition of $\{J_k\}$ that $\Phat(x|x)=0$.
For $x,y\in S$ with $y\not=x$, we compute that
$$
\Phat(y|x)
\ = \ \P[J_{k+1} = y \, \bigm| \, J_k=x]
\ = \ \P[X_{n+1} = y \, \bigm| \, X_n=x, \ X_{n+1}\not=X_n]
$$
$$
\ = \ { \P[X_{n+1} = y, \ X_{n+1}\not=X_n \, \bigm| \, X_n=x]
  \over \P[X_{n+1}\not=X_n \, \bigm| \, X_n=x] }
\ = \ { P(y|x)
  \over \sum_{z\not=x} P(z|x) }
\, ,
$$
as claimed.
\end{proof}

\begin{proposition}\label{multmeanprop}
The conditional distribution of $M_k$ given $J_k$ is equal to the
distribution of $1+G$ where $G$ is a
geometric random variable with success probability $p=\alpha(J_k)$,
i.e.\
\begin{equation}\label{multprobs}
\P[M_k=m \, \bigm| \, J_k]
= (1-p)^{m-1} p
= (1-\alpha(J_k))^{m-1} \alpha(J_k)
\, ,
\quad m=1,2,\ldots
\, ,
\end{equation}
and furthermore $\E[M_k \, | \, J_k] = 1/p = 1 / \alpha(J_k)$.
\end{proposition}

\begin{proof}
If the original chain is at state $x$, then it has probability
$p=\alpha(x)$ of leaving $x$ on the next step, or probability
$1-\alpha(x)$ of remaining at $x$.  Hence, the probability that it will
remain at $x$ for $m$ steps total (i.e., $m-1$ additional steps),
and then leave at the next step, is equal to $(1-p)^{m-1} p$, as claimed.
\end{proof}

%

\begin{proposition}\label{jumpirredprop}
If the original chain $P$ is irreducible, then so is the jump chain $\Phat$.
\end{proposition}

\begin{proof}
Let $x,y\in S$.  Since $P$ is irreducible, there is a path
$x=x_0,x_1,x_2,\ldots,x_m=y$ with $P(x_{i+1}|x_i)>0$ for all $i$.
Without loss of generality, we can assume the $\{x_i\}$ are all distinct.
But if $P(x_{i+1}|x_i)>0$, then~\eqref{jumptrans} implies that also
$\Phat(x_{i+1}|x_i)>0$.  Hence, $\Phat$ is also irreducible.
\end{proof}

\begin{proposition}\label{jumpstatprop}
If the original chain $P$ has stationary distribution $\pi$,
then the jump chain $\Phat$ has stationary distribution $\pihat$ given by
$\pihat(x) = c \, \alpha(x) \, \pi(x)$ where
$c = \Big( \sum_y \alpha(y) \, \pi(y) \Big)^{-1}$.
\end{proposition}

\begin{proof}
Recall that on a discrete space,
$\pi$ is stationary for $P$ if and only if
$\sum_x \pi(x) \, P(y|x) = \pi(y)$ for all $y\in S$.
In that case, we compute that
$$
\sum_x \pihat(x) \, \Phat(y|x)
\ = \ \sum_x \left( c \, \alpha(x) \, \pi(x) \right)
\left( [{P(y|x) / \alpha(x)}] \, \one_{y\not= x} \right)
$$
$$
\ = \ c \, \sum_{x\not=y} \pi(x) \, P(y|x)
\ = \ c \, \left( \sum_{x} \pi(x) \, P(y|x) \right) - c \, \pi(y) \, P(y|y)
$$
$$
\ = \ c \, \pi(y) - c \, \pi(y) \, P(y|y)
\ = \ c \, \pi(y) [1 - P(y|y)]
\ = \ c \, \pi(y) \, \alpha(y)
\ = \ \pihat(y)
\, ,
$$
so that $\pihat$ is stationary for $\Phat$, as claimed.
\end{proof}

\begin{remark}\rm\label{jumprefremark}
Most of the results presented in this section are already known
in the Metropolis-Hastings (reversible) context:
the geometric distribution of the holding times is noted in
Lemma~1(3) of \cite{douc} and
Proposition~1(a) of \cite{iliopoulos};
the modified transition probabilities of the jump chain are stated in
Proposition~1(b) of \cite{iliopoulos};
and the relationship between the stationary distributions of the original
and jump chains is used in
Lemma~1(4) of \cite{douc},
Proposition~1(c) of \cite{iliopoulos}
(see also Proposition~2.1 of \cite{malefaki}),
Lemma~1 of \cite{doucet},
and Section~2 of \cite{deligiannidis}.
\end{remark}

\begin{remark}\rm
It is common that simple modifications of {\it reversible} chains lead to
simple modifications of their stationary distributions.  For example, if a
reversible chain is restricted to a subset of the state space (so any moves
out of the subset are rejected with the chain staying where it is), then its
stationary distribution is equal to the original
stationary distribution conditional on being in that subset (since
the detailed balance equation still holds on the subset).  However,
that property does not hold without reversibility.
For a simple counter-example, let $S=\{1,2,3\}$,
with $P(2|1) = P(3|2) = P(1|3) = 3/4$,
and $P(3|1) = P(1|2) = P(2|3) = 1/4$.
Then if $C=\{1,2\}$, then the stationary distribution of the original
chain is $(1/3,1/3,1/3)$, but the stationary distribution of
the chain restricted to $C$ is $(1/4,3/4)$.
We were thus surprised that Proposition~\ref{jumpstatprop}
holds even for non-reversible chains.
\end{remark}

\section{Using the Jump Chain for Estimation}

The Jump Chain can be used for estimation, as we now discuss.
This approach has also been taken by others;
see Remarks~\ref{estrefremark} and~\ref{foldrefremark} below.


\begin{theorem}\label{mainone}
Given an irreducible Markov chain $\{X_n\}$ with transition probabilities
$P(y|x)$ and stationary distribution $\pi$ on a state space $S$,
and a function $h:S\to\IR$, suppose we
simulate the jump chain $\{J_k\}$ with the
transition probabilities~\eqref{jumptrans}, and then simulate the multiplicities
list $\{M_k\}$ from the conditional probabilities~\eqref{multprobs}
where $p=\alpha(J_k)$ with $\alpha$ as in~\eqref{alpha},
and set 
\begin{equation}\label{ebareqn}
\ebar_L \ = \ {\sum_{k=1}^L M_k \, h(J_k) \over \sum_{k=1}^L M_k}
\, .
\end{equation}
Then $\ebar_L$ is a consistent estimator of the expected value
$\E_\pi(h)$, i.e.\ $\liml_{L\to\infty} \ebar_L = \E_\pi(h)$ w.p.~1.
\end{theorem}

\begin{proof}
Recall (e.g.~\cite{MT}) that the usual estimator
$\ehat_K = {1 \over K} \sum_{n=1}^K h(X_n)$
is consistent, i.e.\ $\lim_{K\to\infty} \ehat_K = \E_\pi(h)$ w.p.~1.
Then, it is seen that
$$
\ebar_L
\ = \ {\sum_{k=1}^L M_k \, h(J_k) \over \sum_{k=1}^L M_k}
\ = \ \ehat_{\sum_{k=1}^L M_k}
\ = \ \ehat_{K(L)}
$$
where $K(L) = \sum_{k=1}^L M_k$.
Since each $M_k \ge 1$, $\lim_{L\to\infty} K(L) = \infty$, so
$\lim_{L\to\infty} \ebar_L
= \lim_{L\to\infty} \ehat_{K(L)}
= \lim_{K\to\infty} \ehat_K
= \E_\pi(h)$ w.p.~1,
as claimed.
\end{proof}

\begin{remark}\rm\label{estrefremark}
The consistency of the estimate~\eqref{ebareqn},
and similarly those of Theorems~\ref{maintwo} and~\ref{contthm} below,
is already known in the Metropolis-Hastings (reversible) context;
see equation~(3) of \cite{malefaki},
Section~2 of \cite{douc},
and equation~(2) of \cite{iliopoulos}.
\end{remark}


On the other hand, combining the Markov chain Law of Large Numbers 
with Propositions~\ref{jumpirredprop} and~\ref{jumpstatprop}
immediately gives:

\begin{proposition}\label{ISprop}
Under the above assumptions,
if we simulate the jump chain $\{J_k\}$ with the
transition probabilities~$\Phat$, then for any
function $g:S\to\IR$ with $\pihat|g|<\infty$, we have
$$
\lim_{L\to\infty} {1 \over L} \sum_{k=1}^L g(J_k)
\ = \ \pihat(g)
\ := \ \sum_{x\in S} g(x) \, \pihat(x)
\ = \ c \, \sum_{x\in S} g(x) \, \alpha(x) \, \pi(x)
\ \ w.p.~1.
$$
\end{proposition}


\begin{corollary}\label{withccor}
Under the above assumptions,
if we simulate the jump chain $\{J_k\}$ with the
transition probabilities~$\Phat$, then for any
function $h:S\to\IR$ with $\pi|h|<\infty$, we have
$$
\lim_{L\to\infty} {1 \over c \, L} \sum_{k=1}^L [h(J_k) / \alpha(J_k)]
\ = \ \pi(h)
\ := \ \sum_{x\in S} h(x) \, \pi(x)
\, ,
\ \ w.p.~1.
$$
\end{corollary}

\begin{proof}
Let $g(x) = h(x) / c \, \alpha(x)$.  Then since
$\pi|h|<\infty$, we have
$$
\pihat|g|
\ = \ \sum_x |g(x)| \, \pihat(x)
\ = \ \sum_x [|h(x)| / c \, \alpha(x)] \, c \, \alpha(x) \, \pi(x)
\ = \ \sum_x |h(x)| \, \pi(x)
\ = \ \pi|h|
\ < \ \infty
\, .
$$
So, the result follows upon plugging this $g$ into
Proposition~\ref{ISprop}.
\end{proof}


%
%

We then have:

\begin{theorem}\label{maintwo}
Under the above assumptions,
if we simulate the jump chain $\{J_k\}$ with the
transition probabilities~$\Phat$, then for any
function $h:S\to\IR$ with $\pi|h|<\infty$, we have
\begin{equation}\label{ealphaeqn}
\lim_{L\to\infty}
{\sum_{k=1}^L [h(J_k) / \alpha(J_k)]
\over \sum_{k=1}^L [1 / \alpha(J_k)]}
\ = \ \pi(h)
\, ,
\ \ w.p.~1.
\end{equation}
\end{theorem}

\begin{proof}
Setting $h \equiv 1$ in Corollary~\ref{withccor} gives that w.p.~1,
$\lim_{L\to\infty} {1 \over c \, L} \sum_{k=1}^L [1 / \alpha(J_k)]
\, = \, 1$.  We then compute that
$$
\lim_{L\to\infty}
{\sum_{k=1}^L [h(J_k) / \alpha(J_k)]
\over \sum_{k=1}^L [1 / \alpha(J_k)]}
\ = \
\lim_{L\to\infty}
{{1 \over cL} \sum_{k=1}^L [h(J_k) / \alpha(J_k)]
\over {1 \over cL} \sum_{k=1}^L [1 / \alpha(J_k)]}
\ = \
{\sum_{x\in S} h(x) \, \pi(x)
\over 1}
\ = \ \pi(h)
\, ,
$$
as claimed.
\end{proof}

Comparing Theorems~\ref{mainone} and~\ref{maintwo}, we see that they
coincide except that each multiplicity random variable $M_k$ has been
replaced by its mean $1/\alpha(J_k)$, cf.\ Proposition~\ref{multmeanprop}.

Finally, we note that although our computer hardware does not allow us to
exploit it, most of the above carries over to Markov chains with densities
on general (continuous) state spaces, as follows.  (The proofs are very
similar to the discrete case, and are thus omitted.)

\begin{theorem}\label{contthm}
Let $\X$ be a general state space, and $\mu$ an atomless
$\sigma$-finite reference
measure on $\X$.  Suppose a Markov chain on $\X$ has transition
probabilities $P(x,dy) = r(x) \, \delta_x(dy) +
\rho(x,y) \, \mu(dy)$ for some $r:\X\to[0,1]$ and
$\rho:\X\times\X\to[0,\infty)$ with $r(x) + \int \rho(x,y) \, \mu(dy) = 1$
for each $x\in\X$, where $\delta_x$ is a point-mass at $x$.
Again let $\Phat$ be the transitions for the corresponding jump chain
$\{J_k\}$ with multiplicities $\{M_k\}$.  Then:
\hfil\break(i)
$\Phat(x,\{x\})=0$, \
and for $x\not=y$, \
$\Phat(x,dy) = {\rho(x,y) \over \int \rho(x,z) \, \mu(dz)} \mu(dy)$.
\hfil\break(ii)
The conditional distribution of $M_k$ given $J_k$ is equal to the
distribution of $1+G$ where $G$ is a
geometric random variable with success probability $p=\alpha(J_k)$
where $\alpha(x) = \P[X_{n+1} \not= x \, | \, X_n=x]
= \int \rho(x,z) \, \mu(dz)
= 1 - r(x)
= 1 - P(x|x)$.
\hfil\break(iii)
If the original chain is $\phi$-irreducible (see e.g.\ \cite{MT})
for some positive $\sigma$-finite measure $\phi$ on $\X$,
then the jump chain is also $\phi$-irreducible for the same $\phi$.
\hfil\break(iv)
If the original chain has stationary distribution $\pi(x) \, \mu(dx)$,
then the jump chain has stationary distribution given by
$\pihat(x) = c \, \alpha(x) \, \pi(x) \, \mu(dx)$ where
$c^{-1} = \int \alpha(y) \, \pi(y) \, \mu(dy)$.
\hfil\break(v)
If $h:\X\to\IR$ has finite expectation, then with probability~1,
$$
\liml_{L\to\infty} {\sum_{k=1}^L M_k \, h(J_k) \over \sum_{k=1}^L M_k}
\ = \ \liml_{L\to\infty} {\sum_{k=1}^L [h(J_k) / \alpha(J_k)]
\over \sum_{k=1}^L [1 / \alpha(J_k)]}
\ = \ \pi(h)
\ := \ \int h(x) \, \pi(x) \, \mu(dx)
\, .
$$
\end{theorem}

\subsection{Application to the Metropolis Algorithm}

Suppose now that the original chain $\{X_n\}$ is a Metropolis algorithm,
with proposal probabilities $Q(y|x)$ which are symmetric
(i.e.\ $Q(y|x)=Q(x|y)$).
Then for $x\not=y$,
$P(y|x) = Q(y|x) \, \min\Big(1, \ {\pi(y) \over \pi(x)}\Big)$.
Hence, by~\eqref{jumptrans}, the jump chain transition probabilities
have $\Phat(x|x)=0$ and for $x\not=y$ are given by
\begin{equation}\label{metjumptrans}
\Phat(y|x)
\ := \ \P[J_1 = y \, | \, J_0=x]
\ = \ {  Q(y|x) \, \min\Big(1, \ {\pi(y) \over \pi(x)}\Big)
\over \sum_{z \not= x} Q(z|x) \, \min\Big(1, \ {\pi(z) \over \pi(x)}\Big) }
\, .
\end{equation}
Also, here
\begin{equation}\label{metalpha}
\alpha(x)
\ = \ \sum_{y \not= x} P(y|x)
\ = \ \sum_{y \not= x} Q(y|x) \, \min\Big(1, \ {\pi(y) \over \pi(x)} \Big)
\, .
\end{equation}

%

A special case is where the proposal probabilities $Q(x,\cdot)$
are {\it uniform} over all ``neighbours'' of $x$, where each state has the
same number $N$ of neighbours.  We assume that
$x$ is {\it not} a neighbour of itself,
and that $x$ is a neighbour of $y$ if and only if $y$ is a neighbour of $x$.
Then for $x\not=y$, $P(y|x) = \ooN \, \min\Big(1, \ {\pi(y) \over
\pi(x)}\Big)$.
And, by~\eqref{jumptrans}, the jump chain transition probabilities
have $\Phat(x|x)=0$ and for $x\not=y$ are given by
\begin{equation}\label{unifjumptrans}
\Phat(y|x)
\ = \ {  \min\Big(1, \ {\pi(y) \over \pi(x)}\Big)
\over \sum_{z \sim x} \min\Big(1, \ {\pi(z) \over \pi(x)}\Big) }
\end{equation}
where the sum is over all neighbours $z$ of $x$.
Also, here
\begin{equation}\label{unifalpha}
\alpha(x)
\ = \ \ooN \, \sum_{y \not= x} \min\Big(1, \ {\pi(y) \over \pi(x)} \Big)
\, .
\end{equation}

The use of the estimators \eqref{ebareqn} and \eqref{ealphaeqn} in the
context of uniform Metropolis algorithms can be carried out very
efficiently using special parallelised computer hardware (see e.g.\
Section~\ref{sec-numerical} below), and was our original motivation for
this investigation.


\begin{remark}\rm\label{foldrefremark}
The ``$n$-fold way'' of Bortz et al.\
\cite{bortz} considers the Ising model, and selects
the next site to flip proportional to its probability of flipping, by
first classifying all sites in terms of their spin and neighbour counts.
This creates a rejection-free Metropolis-Hastings
algorithm in the same spirit as
our approach, though specific to the Ising model.
Later authors parallelised their algorithm, still for the
Ising model; see e.g.\ \cite{lubachevsky} and \cite{korniss}. 
\end{remark}

\section{Alternating Chains}

Sometimes we have two or more different Markov chains and we wish to
alternate between them in some pattern.
And, we might wish to use rejection-free
sampling for some or all of the individual chains.
However, if this is done naively, it can lead to bias:


\medskip\noindent {\bf Example~3: }
Let $S=\{1,2,3,4\}$, and $\pi=(1-\epsilon,3\epsilon,1-\epsilon,1-\epsilon)/3$
for some small positive number $\epsilon$ (e.g.\ $\epsilon=0.001$).  Let
$Q_1(x, \, x+1)=Q_1(x, \, x-1)=1/2$ and
$Q_2(x, \, x+1)=Q_2(x, \, x+2)=Q_2(x, \, x-1)=Q_2(x, \, x-2)=1/4$
be two different
proposal kernels, and let $P_1$ and $P_2$ be usual Metropolis
algorithms for $\pi$ with proposals $Q_1$ and $Q_2$ respectively.
Then, each of $P_1$ and $P_2$ will converge to $\pi$, as will the
algorithm of alternating between $P_1$ and $P_2$ any fixed number of
times.  However, if we instead alternate between doing one
{\it jump} step of $P_1$ and then one {\it jump} step of $P_2$,
then this combined chain will not converge to the correct distribution.
Indeed, the corresponding escape probabilities $\alpha_1(x)$ and
$\alpha_2(x)$ are all reasonably large (at least 1/4) except for
$\alpha_1(1)=\epsilon/2$ which is extremely small.  This means that when our
algorithm uses $P_1$ from state~1 then it will have an extremely large
multiplicity $M_k$ which will lead to extremely large weight of the
state~1.  Indeed,
if we use the alternating
jump chains algorithm, then the estimators $\ebar_L$ as
in~\eqref{ebareqn} will have the property that
as $\epsilon \searrow 0$, their limiting value converges to
$h(1)$ instead of $\pi(h)$, i.e.\
$$
\lim_{\epsilon \searrow 0} \lim_{L\to\infty} \ebar_L \ = \ h(1)
\, .
$$
Hence, convergence to $\pi$ fails in this case.
\qed

However, this convergence problem can be fixed if we control the number of
effective repetitions of each kernel.  Specifically, suppose
we choose in advance some number $L_0$ of effective repetitions we wish to
perform for the kernel $P_1$ before switching to the kernel $P_2$.
Then we can do this in a rejection-free manner as follows:

\noindent {\bf 1.} \
Set the number of remaining repetitions, $L$, equal to some fixed initial
value $L_0$.

\noindent {\bf 2.} \
Find the next jump chain value $J_k$ and multiplicity $M_k$
corresponding to the Markov chain $P_1$, as above.

\noindent {\bf 3.} \
If $M_k \ge L$, then replace $M_k$ by $L$, and keep $J_k$ as it is,
and include that $J_k$ and $M_k$ in the estimate.  Then,
return to step~{\bf 1} with the next kernel $P_2$.

\noindent {\bf 4.} \
Otherwise, if $M_k < L$, then keep $M_k$ and $J_k$ as they are, and count
them in the estimate, and then replace $L$ by $L-M_k$ and return to
step~{\bf 2} with the same kernel $P_1$.


\medskip

This modified algorithm is equivalent to applying the original
(non-rejection-free) kernel $P_1$ a total of
$L_0$ times before switching to the next kernel $P_2$.
As such, it has no bias,
and is consistent and will converge to the correct distribution without
any errors as in the counter-example above.

\section{Application to Parallel Tempering}

Parallel tempering (or, replica exchange) \cite{swendsen, geyer}
proceeds by considering different versions of the target distribution
$\pi$ powered by different inverse-temperatures $\beta$, of the
form $\pi^{(\beta)}(x) \propto \big(\pi(x)\big)^\beta$.
It runs separate MCMC algorithms on each target $\pi^{(\beta)}$, for some
fixed number of iterations, and then
proposes to ``swap'' pairs of values $X^{(\beta_1)} \leftrightarrow
X^{(\beta_2)}$.  This swap proposal is
accepted with the usual Metropolis algorithm probability
\begin{equation}\label{genswapprob}
\min\left[1, \ {
     \pi^{(\beta_1)}(X^{(\beta_2)}) \, \pi^{(\beta_2)}(X^{(\beta_1)})
      \over \pi^{(\beta_1)}(X^{(\beta_1)}) \, \pi^{(\beta_2)}(X^{(\beta_2)})
}\right]
\end{equation}
which preserves the product target measure $\prod_\beta \pi^{(\beta)}$.

But suppose we instead want to run parallel tempering using jump chains,
i.e.\ using a rejection-free algorithm within each temperature.
If we run a fixed number of
rejection-free moves of each within-temperature chain, followed by one
``usual'' swap move, then this can lead to bias, as the following example
shows.

\medskip\noindent {\bf Example~4: }
Let $S=\{1,2,3\}$, with $\pi(1)=\pi(3)=1/4$ and $\pi(2)=1/2$.
Suppose there are just two inverse-temperature values, $\beta_0=1$ and
$\beta_1=5$.  Suppose each within-temperature chain proceeds as a
Metropolis algorithm, with proposal distribution given by $Q(y|x)=1/2$
whenever $y\not=x$.  (That is, we can regard the three states of $S$ as
being in a circle, and the chain proposes to move one step clockwise or
counter-clockwise with probability 1/2 each, and then accepts or rejects
this move according to the usual Metropolis procedure.)
If we run a usual parallel tempering algorithm, then the within-temperature
moves will converge to the corresponding stationary distributions
$\pi^{(0)} = \pi = (1/4,1/2,1/4)$ and
$\pi^{(5)} = (1/34,32/34,1/34)$ respectively.
Then, given current chain values $X^{(0)}$ and $X^{(5)}$,
if we attempt a usual swap move, it will be accepted with probability
\begin{equation}\label{swapprob}
\min\left[1, \ {
     \pi^{(0)}(X^{(5)}) \ \pi^{(5)}(X^{(0)})
          \over \pi^{(0)}(X^{(0)}) \ \pi^{(5)}(X^{(5)}) }\right]
\, .
\end{equation}
These steps will all preserve the product stationary distribution
$\pi^{(0)} \times \pi^{(5)}$, as they should.
However, if we instead run a rejection-free within-temperature chain,
then convergence fails.  Indeed,
from each state the jump chain is equally likely to move
to either of the other two states, so each jump chain
will converge to the {\it
uniform} distribution on $S$.  The acceptance probability~\eqref{swapprob}
will then lead to incorrect distributional convergence,
e.g.\ if $X^{(0)}=2$ and $X^{(5)}=3$, then a proposal
to swap $X^{(0)}$ and $X^{(5)}$ will always be accepted, leading
to an excessively large probability that $X^{(0)}=3$.
Indeed, in simulations\footnote{Performed using the R program
available at: http://probability.ca/rejectionfreesim}
the fraction of time that $X^{(0)}=3$ right after a swap proposal is about
44\%, much larger than the 1/3 probability it should be.
\qed

To get rejection-free parallel tempering to converge
correctly, we recall from
Proposition~\ref{jumpstatprop} that the
rejection-free chains actually converge to the modified stationary
distributions $\pihat$, not $\pi$.
We should thus modify the acceptance probability~\eqref{genswapprob} to:
$$
\min\left[1, \ {
     \pihat^{(\beta_1)}(X^{(\beta_2)}) \ \pihat^{(\beta_2)}(X^{(\beta_1)})
  \over \pihat^{(\beta_1)}(X^{(\beta_1)}) \ \pihat^{(\beta_2)}(X^{(\beta_2)})
}\right]
\qquad\qquad\qquad\qquad\qquad\qquad\qquad\qquad
$$
\begin{equation}\label{modgenswapprob}
\ = \ \min\left[1, \ {
     \alpha^{(\beta_1)}(X^{(\beta_2)}) \ \pi^{(\beta_1)}(X^{(\beta_2)})
	\ \alpha^{(\beta_2)}(X^{(\beta_1)}) \ \pi^{(\beta_2)}(X^{(\beta_1)} )
      \over \alpha^{(\beta_1)}(X^{(\beta_1)}) \ \pi^{(\beta_1)}(X^{(\beta_1)})
	\ \alpha^{(\beta_2)}(X^{(\beta_2)}) \ \pi^{(\beta_2)}(X^{(\beta_2)} )
}\right]
\, .
\end{equation}
Such swaps will preserve the product modified stationary distribution
$\prod_\beta \pihat^{(\beta)}$,
rather than trying to preserve the unmodified stationary distribution
$\prod_\beta \pi^{(\beta)}$.
(If necessary, the escape probabilities $\alpha(x)$ can be estimated from
a preliminary run.)
The rejection-free
parallel tempering algorithm will thus converge to
$\prod_\beta \pihat^{(\beta)}$,
thus still allowing for valid inference as in
Theorems~\ref{mainone} and~\ref{maintwo}.

\medskip\noindent {\bf Example~4 (continued): }
In this example,
$\alpha^{(0)}(1) = \alpha^{(0)}(3) = \alpha^{(5)}(1) = \alpha^{(5)}(3) = 1$,
$\alpha^{(0)}(2) = 1/2$, and $\alpha^{(5)}(2) = 1/32$.
So, if $X^{(0)}=2$ and $X^{(5)}=3$, then
according to~\eqref{modgenswapprob}, a proposal
to swap $X^{(0)}$ and $X^{(5)}$ will be accepted with probability
$$
\min\left[1, \ {
     \alpha^{(0)}(X^{(5)}) \ \pi^{(0)}(X^{(5)})
	\ \alpha^{(5)}(X^{(0)}) \ \pi^{(5)}(X^{(0)} )
          \over \alpha^{(0)}(X^{(0)}) \ \pi^{(0)}(X^{(0)})
	\ \alpha^{(5)}(X^{(5)}) \ \pi^{(5)}(X^{(5)} )
}\right]
\qquad\qquad\qquad\qquad
$$
$$
\ = \
\min\left[1, \ { (1) (1/4) (1/32) (1/2) \over (1/2) (1/2) (1) (1/34) } \right]
	\ = \ 34/64 \ = \ 17/32
\, ,
$$
and such swaps will instead preserve the product stationary distribution
$\pihat^{(0)} \times \pihat^{(5)}$.
Indeed, in simulations\footnote{Performed using the R program
available at: http://probability.ca/rejectionfreemod}
the fraction of time that $X^{(0)}=3$ right after a swap proposal with
this modified acceptance probability becomes about 1/3, as it should be.
\qed

%

\iffujitsu

\section{How to Sample Proportionally}

{\bf [NOTE: This section can be removed if Fujitsu wishes.]}

Given $A_i>0$ for $i=1,2,\ldots,N$, how can we sample $Z$ so that
$\P(Z=i) = A_i \bigm/ \sum_j A_j$?

We could choose $U \sim \Uniform[0,1]$, and then set $Z = \min\{i :
\sum_{j=1}^i A_j > U \sum_{j=1}^N A_j \}$.  But this involves summing all
of the $A_j$, which is inefficient.

If $\sum_j A_j = 1$, then we could choose $U \sim \Uniform[0,1]$ and just
set $Z = \min\{i : \sum_{j=1}^i A_j > U\}$.  This is slightly easier, and
can be done by binary searching.  But it still requires summing lots of the
$A_j$, which could still be inefficient.

If $\sum_j A_j < 1$, then we could choose $U \sim \Uniform[0,1]$, and then
still set $Z = \max\{i : \sum_{j=1}^i A_j > U\}$, except if no such $i$
exists then we {\it reject} that choice of $U$ and start again.  But in
addition to the previous problems, this could involve lots of
{\it rejection} if $\sum_j A_j$ is much smaller than~1, which is again
inefficient.

Another option is the following method, based on Efraimidis and
Spirakis~\cite{efraimidis}; see also the {\it $n$-fold way} approach to
{\it kinetic Monte Carlo} in Bortz et al.~\cite{bortz}.

\begin{proposition}\label{kineticprop}%
Let $A_1,A_2,\ldots,A_N$ be positive numbers.
Let $\{R_j\}_{j=1}^N$ be i.i.d.~$\sim \Uniform[0,1]$, and let
$d_j = - \log(R_j)/A_j$ for $j=1,2,\ldots,N$.
Finally, set $Z = \argmin_j \, d_j$.
Then $\P[Z=i] = A_i / \sum_j A_j$, i.e.\
$Z$ selects $i$ from $\{1,2,\ldots,N\}$ with probability proportional
to $A_i$.
(This is useful since \argmin\ can be computed
much faster than dividing by the sums. [EXPLAIN?])
\end{proposition}

\begin{proof}
We compute that
$$
\P[Z=i]
\ = \ \P[d_j>d_i \ \forall j\not=i ]
\qquad\qquad\qquad\qquad\qquad\qquad\qquad\qquad\qquad
$$
$$
\ = \ \P\Big[ - \log(R_j)/A_j > - \log(R_i)/A_i \ \forall j\not=i \Big]
\ = \
\P\Big[ R_j < R_i^{A_j/A_i} \ \forall j\not=i \Big]
$$
$$
\ = \ \int_0^1
\P\Big[ R_j < R_i^{A_j/A_i} \ \forall j\not=i \, \Big| \, R_i=x\Big] \ dx
\ = \ \int_0^1 \P\big[ R_j < x^{A_j/A_i} \ \forall j\not=i \big] \ dx
$$
$$
\ = \ \int_0^1 \prod_{j\not=i} x^{A_j/A_i} \ dx
\ = \ \int_0^1 x^{\big(\sum_{j\not=i} A_j\big)/A_i} \ dx
\ = \ {x^{\Big(\big(\sum_{j\not=i} A_j\big)/A_i\Big)+1}
	\over \Big(\big(\sum_{j\not=i} A_j\big)/A_i\Big)+1} \Biggm|_{x=0}^{x=1}
$$
$$
\ = \ {1 \over \Big(\big(\sum_{j\not=i} A_j\big)/A_i\Big)+1}
\ = \ {A_i \over \Big(\sum_{j\not=i} A_j\Big)+A_i}
\ = \ {A_i \over \sum_{j=1}^N A_j}
$$
as required.
\end{proof}


\fi

\section{Numerical Examples}
\label{sec-numerical}

In this section, we introduce applications and simulations to illustrate
the advantage of Rejection-Free algorithm. We compare the efficiency of
the Rejection-Free and standard Metropolis algorithms in three different
examples. The first example is a Bayesian inference model on a real data
set taken from the Education Longitudinal Study of~\cite{ELS:2002}.
The second example involves sampling from
a two-dimensional ferromagnetic $4 \times4$ Ising
model.  The third example is a pseudo-marginal~\cite{andrieu} version
of the Ising model.  All three simulations show that the introduction
of the Rejection-Free method leads to
significant speedup.  This provides concrete numerical evidence for
the efficiency of using the rejection-free approach to improve the
convergence to stationarity of the algorithms.

\subsection{A Bayesian Inference Problem with Real Data}

For our first example, we consider the
Education Longitudinal Study~\cite{ELS:2002}
real data set consisting of final course grades of over 9,000
students. We take a random subset of 200 of these
9,000 students, and denote their scores as $x_1, x_2, \dots, x_{200}$.
(Note that all scores in this data set are integers between 0 and 100.)

Our parameter of interest $\theta$ is the true average value of the
final grades for these 200 students, rounded to 1 decimal place (so
$\theta \in \{0.1, 0.2, 0.3, \dots, 99.7, 99.8, 99.9\}$ is still discrete,
and can be studied using specialised computer hardware).
The likelihood function for this model is the binomial distribution
\begin{equation}\label{example1likelihood}
\textbf{L}(x|\theta) \ = \  {100 \choose x} \, \theta^x \, (1-\theta)^{100-x}
\, .
\end{equation}
For our prior distribution, we take
\begin{equation}\label{example1prior}
    \theta \ \sim \ \Uniform\{0.1, 0.2, 0.3, \dots,
							99.7, 99.8, 99.9\}
\, .
\end{equation}
The posterior distribution $\pi(\theta)$ is then proportional to
the prior probability function~\eqref{example1prior}
times the likelihood function~\eqref{example1likelihood}.

We ran an Independence Sampler for this posterior distribution,
with fixed proposal
distribution equal to the prior~\eqref{example1prior}, either with or
without the Rejection-Free modification.  For each of these two algorithms,
we calculated the effective sample size, defined as
$$
    \textit{ESS}(\theta) = \frac{N}{1+ 2 \sum_{k=1}^\infty \rho_k(\theta)}
\, ,
$$
where N is the number of posterior samples,
and $\rho_k(\theta)$ represents
autocorrelation at lag k for the posterior samples of $\theta$. (For
a chain of finite length, the sum $\sum_{k=1}^\infty \rho_k(\theta)$
cannot be taken over all $k$, so instead we just sum until the values
of $\rho_k(\theta)$ become negligible.)
For fair comparison, we consider both the ESS per
iteration, and the ESS per second of CPU time.

\begin{table}[]
    \centering
    \begin{tabular}{|c|c|c|}
        \hline
         &  ESS per Iteration & ESS per CPU Second\\
        \hline
        Metropolis  & 0.0074 & 47 \\
        \hline
        Rejection-Free & 0.9100 & 4,261 \\
        \hline
        \hline
        Ratio & 123.0 & 90.7 \\
        \hline
    \end{tabular}
    \caption{Median of Effective Sample Sizes from 100 Runs each
of the Metropolis and Rejection-Free algorithms}
    \label{table2}
\end{table}

Table~\ref{table2} presents
the median ESS per iteration, and median ESS per second, from 100
runs of $100,000$ iterations each, for each of the two algorithms.
We see from Table~\ref{table2} that Rejection-Free
outperforms the Metropolis algorithm by a factor of approximately one
hundred, in terms of both ESS per iteration and ESS per second.
This clearly illustrates the efficiency of the Rejection-Free algorithm.



\subsection{Simulations of an Ising Model}

We next present a simulation study of a ferromagnetic Ising model on
a two-dimensional $4 \times 4$ square latice. The energy function for
this model is given by
$$
    E(S) = - \sum_{i<j}J_{i j} s_i s_j
\, ,
$$
where each spin $s_i,s_j \in \{-1,1\}$,
and $J_{i j}$ represents the interaction between the $i^{\rm
th}$ and $j^{\rm th}$ spins. To make
only the neighbouring spins in the lattice
interact with each other, we take $J_{i j} = 1$ for all neighbours $i$
and $j$, and $J_{i j}$ = 0 otherwise.
The Ising model then has probability distribution
propositional to the exponential of the energy function:
$$
\Pi(S) \ \propto \ \exp[-E(S)]
\, .
$$

We investigate the efficiency of the samples produced in four different
scenarios: Metropolis algorithm and Rejection-Free, both with and
without Parallel Tempering.  For the Parallel Tempering versions, we set
$$
\Pi_T(S) \ \propto \ \exp[-E(S) \, / \, T]
\, .
$$
Here $T = 1$ is the
temperature of interest (which we want to sample from).
We take $T = 2$ as the highest
temperature, since when $T = 2$ the probability
distribution for magnetization is quite flat (with highest probability
$\textbf{P}[M(S) = 14] = 0.083$, and lowest probability
$\textbf{P}[M(S) = 2] = 0.037$).
Including the one additional temperature $T =
\sqrt{2}$ gives three temperatures $1,
\sqrt{2}, 2$ in geometric progression,
with average swap acceptance rate
$31.6\%$ which is already higher than the $23.4\%$ recommended
in~\cite{sectemper}, indicating that three temperatures is enough.

We study convergence of the {\it magnetization} value, where
the magnetization of a given state $S$ of the Ising model is defined as: 
$$
M(S) \ = \ \sum_i s_i
$$
For our $4 \times 4$ Ising model, 
$$
M(S) \in \mathcal{M} = \{-16, -14, -12, \dots, -2, 0, 2, \dots, 12, 14, 16\}
\, .
$$
We measure the distance to stationarity
by the {\it total variation distance}
between the sampled and the actual magnetization distributions
after $n$ iterations, defined as:
$$
    \textrm{TVD}(n) \ = \
\frac{1}{2} \sum_{m \in \mathcal{M}} \Big| \P[M(X_n)=m]
			- \Pi\{S : M(S)=m\} \Big|
\, ,
$$
where $M(X_n)$ is the magnetization of the chain at iteration~$n$,
and $\Pi\{S : M(S)=m\}$ represents the stationary probability of
magnetization value $m$.  Thus, convergence to stationarity
is described by how quickly TVD$(n)$ decreases to~0.

Figure~\ref{figure1} lists
the average total variation distance TVD$(n)$ for each version,
as a function of the number of iterations~$n$,
based on 100 runs of each of the four scenarios, of $10^6$
iterations each.  It illustrates that, with or without Parallel
Tempering, the use of Rejection-Free provides significant speedup,
and TVD decreases much more rapidly with the Rejection-Free method than
without it. This
provides concrete numerical evidence for the efficiency of using
Rejection-Free to improve the convergence to stationarity of the
algorithm.

\begin{figure}
    \centering
    \includegraphics[scale=0.4]{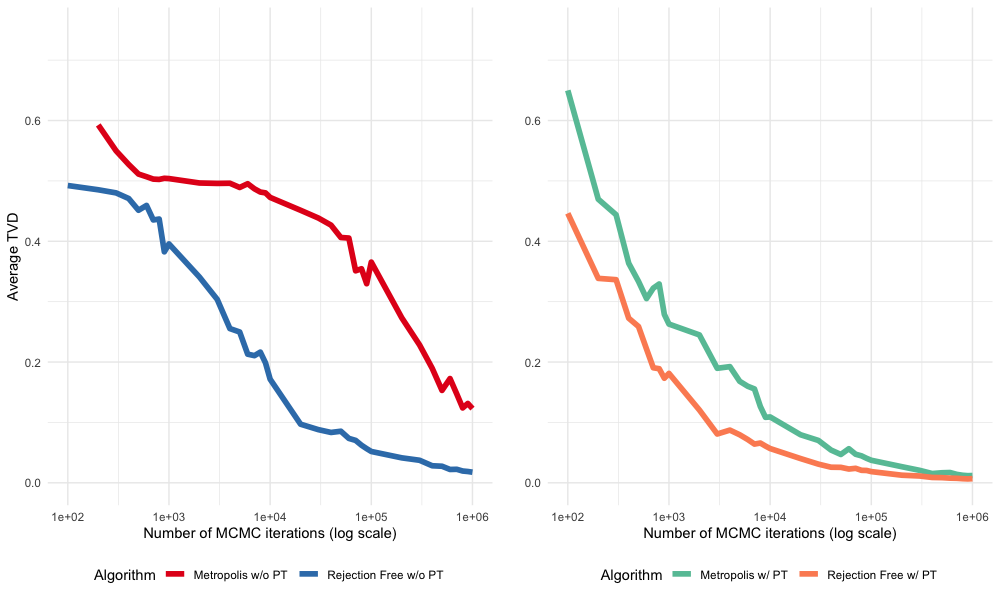}
    \caption{Average total variation distance TVD$(n)$ between sampled and
    actual distributions as a function of the number of iterations, for
    four scenarios: Metropolis versus Rejection-Free both without
    (left) and with (right) Parallel Tempering.}
    \label{figure1}
\end{figure}

We next consider the issue of computational cost.
The Rejection-Free method requires
computing probabilities for all neighbors of the current
state. However, with specialised computer hardware, Rejection-Free can
be very efficient since the calculation of the probabilities for all neighbours
and selection of the next state can both be done in parallel. The
computational cost of each iteration of Rejection-Free is therefore
equal to the maximum
cost used on each neighbor. Similarly, for Parallel Tempering, we can
calculate all of the different temperature chains in parallel.
The average CPU time per iteration for each of the
four different scenarios are presented in Table~\ref{table1}.  It
illustrates that the computational cost of Rejection-Free without
Parallel Tempering was comparable to that of the usual Metropolis
algorithm, though Rejection-Free with Parallel Tempering does require
up to 50\% more time than the other three scenarios.

\begin{table}[]
    \centering
    \begin{tabular}{|c|c|}
        \hline
        Algorithm & Average CPU Time (nanoseconds) \\
        \hline
        Metropolis w/o PT & 420\\
        Rejection-Free w/o PT &  407\\
        Metropolis w/ PT & 463\\
        Rejection-Free w/ PT & 611\\
        \hline
    \end{tabular}
    \caption{Average CPU time per iteration for each of four scenarios:
Metropolis and Rejection-Free, both with and without Parallel Tempering}
    \label{table1}
\end{table}

Figure~\ref{figure2} shows the average total variation
distance as a function of the total CPU time used for
each algorithm. Figure~\ref{figure2} is quite similar to
Figure~\ref{figure1}, and gives the same overall conclusion:
with or without Parallel Tempering, the use of Rejection-Free provides
significant speedup, even when computational cost is taken into account.

\begin{figure}
    \centering
    \includegraphics[scale=0.4]{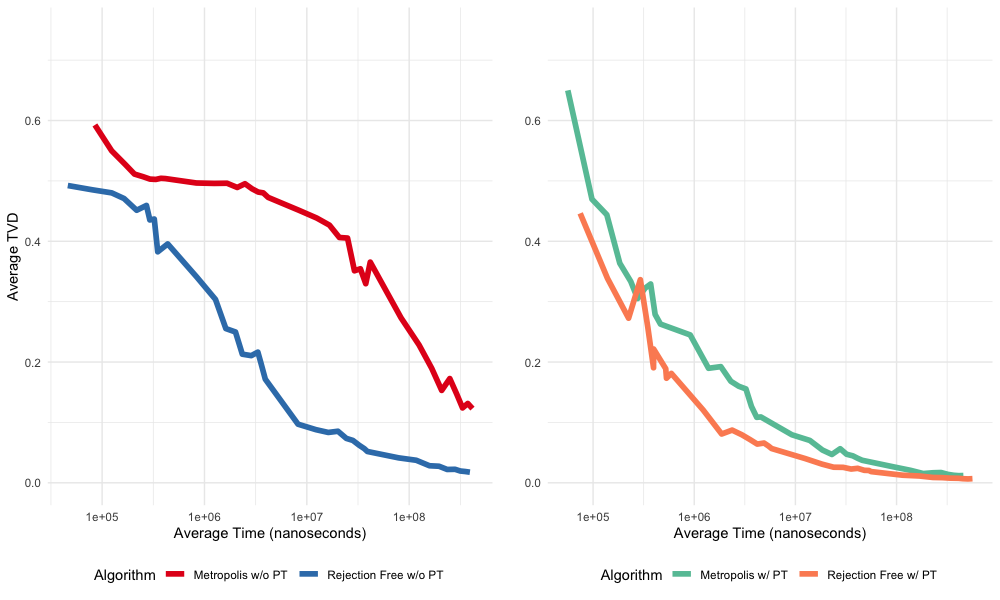}
    \caption{Average total variation distance TVD$(n)$ between sampled and
    actual distributions as a function of CPU time cost, for
    four scenarios: Metropolis versus Rejection-Free both without
    (left) and with (right) Parallel Tempering.}
    \label{figure2}
\end{figure}

As a final check, we also calculated the effective sample size, similar
to the first example. First, we generated 100 MCMC chains of
$100,000$ iterations each, from all four algorithms. Then, we calculated the
effective sample size for each chain, and normalized the results
by either the number of iterations or the total CPU time for each
algorithm. Table~\ref{table3} shows the median of ESS per iteration and
ESS per CPU second. It again illustrates that
Rejection-Free can produce great speedups, increasing the ESS per CPU
second by a factor of over 50 without Parallel Tempering, or a factor of 2
with Parallel Tempering.

\begin{table}[]
    \centering
    \begin{tabular}{|c|c|c|}
        \hline
         &  ESS per Iteration & ESS per CPU Second\\
        \hline
        Metropolis w/0 PT & 0.0003 & 77.76 \\
        \hline
        Rejection-Free w/o PT & 0.0016 & 4,227 \\
        \hline
         Metropolis w/ PT & 0.0057 & 11,830 \\
        \hline
        Rejection-Free w/ PT & 0.0138 & 23,459 \\
        \hline
    \end{tabular}
    \caption{Median of Normalized Effective Sample Sizes for four scenarios:
Metropolis and Rejection-Free, both with and without Parallel Tempering}
    \label{table3}
\end{table}
  
\subsection{A Pseudo-Marginal MCMC Example}

If the target density itself is not available analytically, but an
unbiased estimate exists, then {\it pseudo-marginal MCMC} \cite{andrieu}
can still be used to sample from the correct target distribution.
We next apply the Rejection-Free method to a pseudo-marginal algorithm
to show that Rejection-Free can provide speedups in that case, too.

In the previous example of the $4 \times 4$ Ising model,
the target probability distributions were defined as
$$
    \Pi(S) \ \propto \ \exp\{- \frac{E(S)}{T}\}
\, .
$$
We now pretend that this target density is not available, and we
only have access to an unbiased estimator given by
$$
\Pi_0(S)
\ \propto \ \Pi(S) \times A
\ = \ \exp\{\frac{E(S)}{T}\} \times A
\, ,
$$
where $A \sim \textrm{Gamma}(\alpha = 10, \beta = 10)$ is a random
variable (which is sampled independently
every time as we try to compute the target
distribution).  Note that $\E(A) = 10/10 = 1$,
so $\E[\Pi_0(S)] = \Pi(S)$, and the estimator is unbiased (though $A$ has
variance $10/10^2 = 1/10 > 0$).

Using this unbiased estimate of the target distribution as for
pseudo-marginal MCMC, we again investigated the convergence of samples
produced by the same four scenarios: Metropolis and Rejection-Free,
both with and without Parallel Tempering.  Figure~\ref{figure3} shows
the average total variation distance TVD$(n)$ between the sampled and the
actual magnetization distributions, for 100 chains, as a function of
the iteration~$n$, keeping all the other settings the same as before.
This figure is quite similar
to Figure~\ref{figure1}, again showing that
with or without Parallel Tempering, the use of Rejection-Free provides
significant speedup, even in the pseudo-marginal case.

\begin{figure}
    \centering
    \includegraphics[scale=0.4]{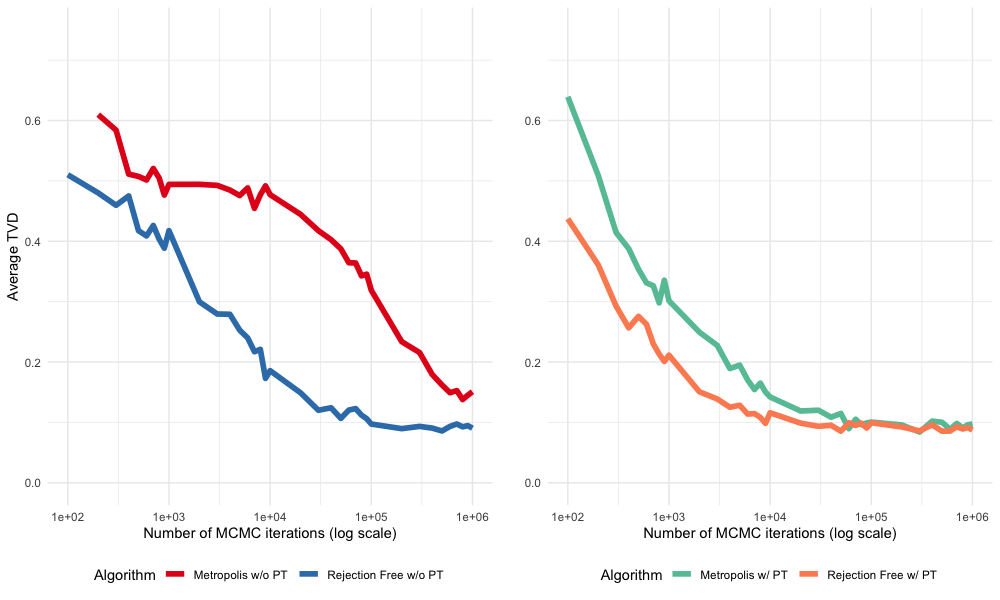}
    \caption{Average of total variation distance between sampled and
    actual distributions as a function of number of iterations for
    probability with noise Gamma(10, 10) in four scenarios: Metropolis
    versus Rejection-Free without Parallel Tempering (left) and with
    Parallel Tempering (right)}
    \label{figure3}
\end{figure}

\section{Summary}

This paper has considered the use of parallelised computer hardware to 
run rejection-free versions of the Metropolis algorithm.  We showed that
the Uniform Selection Algorithm might fail to converge to the correct
distribution or even visit the maximal value.  However, the Jump Chain
with appropriate weightings can provide consistent estimates of expected
values in an efficient rejection-free manner.  Care must be taken when
alternating between multiple rejection-free chains, or when using
rejection-free chains for parallel tempering, but appropriate adjustments
allow for valid samplers in those cases as well.  Simulations of our
methods on several examples illustrate the significant speedups that result
from using the Rejection-Free method to obtain more efficient samples.

\section{Appendix: Proof of Proposition~\ref{appendixprop}}

%

\begin{lemma}\label{locallemma}
For the Uniform Selection chain of Figure~\ref{infuniffig},
let $s(x) = \P($hit 4 before 0$\, | \, X_0=x)$.
Then $s(0)=0$, $s(1)=3/7$, $s(2)=4/7$, $s(3)=13/21$, and $s(4)=1$.
\end{lemma}

\proof
Clearly $s(0)=0$ and $s(4)=1$.
Also, by conditioning on the first step, for $1 \le x \le 3$ we have
$s(x) = p_{x,x-1} \, s(x-1) + p_{x,x+1} \, s(x+1)$.
In particular, $s(1) = (1/4) s(0) + (3/4) s(2) = (3/4) s(2)$,
and $s(2) = (1/4) s(1) + (3/4) s(3)$,
and $s(3) = (8/9) s(2) + (1/9) s(4) = (8/9) s(2) + (1/9)$.
We solve these equations using algebra.
Substituting the first equation into the second,
$s(2) = (1/4)(3/4) s(2) + (3/4) s(3)$, so
$(13/16) s(2) = (3/4) s(3)$, so $s(3) = (13/16)(4/3) s(2) = (13/12) s(2)$.
Then the third equation gives $(13/12) s(2) = (8/9) s(2) + (1/9)$,
so $(7/36) s(2) = (1/9)$, so $s(2)=(1/9)(36/7) = 4/7$.
Then $s(1) = (3/4) s(2) = (3/4)(4/7) = 3/7$,
and $s(3) = (8/9) s(2) + (1/9) = (8/9) (4/7) + (1/9) = 13/21$,
as claimed.
\qed

\begin{lemma}\label{fourslemma}
Suppose the Uniform Selection chain for Example~2 begins at state $x=4a$ for
some positive integer $a$.  Let $C$ be the event that the chain hits
$4(a+1)$ before hitting $4(a-1)$.  Then $q := \P(C) = 9/17 > 1/2$.
\end{lemma}

\proof
By conditioning on the first step, we have that
$$
q
\ = \
\P(C \, | \, X_0=4a)
\qquad\qquad\qquad\qquad\qquad\qquad\qquad\qquad\qquad\qquad\qquad\qquad\qquad
$$
$$
\ = \
\P(X_1=4a+1) \ \P(C \, | \, X_0=4a+1)
+ \P(X_1=4a-1) \ \P(C \, | \, X_0=4a-1)
$$
$$
\ = \
(1/2) \ \P(C \, | \, X_0=4a+1)
+ (1/2) \ \P(C \, | \, X_0=4a-1)
\, .
$$
But from $4a+1$, by Lemma~\ref{locallemma},
we either reach $4a+4$ before returning to $4a$ (and ``win'')
with probability 3/7, or we first return to $4a$ (and ``start over'')
with probability 4/7.
Similarly, from $4a-1$,
we either return to $4a$ (and ``start over'') with probability 13/21,
or we reach $4a-4$ before returning to $4a$ (and ``lose'')
with probability 8/21.
Hence,
$$
q
\ = \
(1/2) \, [ (3/7) + (4/7) q ]
+ (1/2) \, [ (13/21) q + 0]
\, .
$$
That is, $q = (3/14) + (2/7) q + (13/42) q = (3/14) + (25/42) q$.
Hence, $q = (3/14) \bigm/ (17/42) = 9/17 > 1/2$.
\qed

We then have:

\begin{corollary}\label{appendixcor}
Suppose the Uniform Selection chain for Example~2 begins at state $4a \ge
8$ for some positive integer $a \ge 2$.
Then the probability it will ever reach the state~4 is 
$(8/9)^{a-1} < 1$.
\end{corollary}

\proof Consider a sub-chain $\{\Xt_n\}$ of $\{X_n\}$ which just records
new multiples of~4.  That is, if the original chain is at the state $4b$,
then the new chain is at $b$.  Then, we wait until the original reaches
either $4(b-1)$ or $4(b+1)$ at which point the next state of the new chain
is $b-1$ or $b+1$ respectively.  Then Lemma~\ref{fourslemma} says that
this new chain is performing simple random walk on the positive integers,
with up-probability 9/17 and down-probability 8/17.  Then it follows from
the Gambler's Ruin
formula (e.g.~\cite[equation~7.2.7]{grprobbook}) that, starting
from state~$a$, the probability that the new chain will ever reach the state~1 
is equal to $[(8/17)/(9/17)]^{a-1} = (8/9)^{a-1} < 1$, as claimed.
\qed

\medskip
Since the chain starting at $4a$ for $a \ge 2$
cannot reach state~3 without first reaching state~4,
Proposition~\ref{appendixprop}
follows immediately from Corollary~\ref{appendixcor}.

\medskip
If we instead cut off the example at the state $4L$, then the Gambler's
Ruin formula (e.g.~\cite[equation~7.2.2]{grprobbook})
says that from the state $4(L-1)$, the probability of
reaching the state~4 before returning to the state $4L$
is $[(9/8)^1-1] \bigm/ [(9/8)^{L-2}-1] < (8/9)^{L-1}$
(since $[A-1] \bigm/ [B-1] < A/B$ whenever $1<A<B$), so the expected
number of attempts to reach state~4 from state~$4L$ is
more than $(9/8)^{L-1}$.


\bigskip\noindent\bf Acknowledgements. \rm
This work was supported by research grants from Fujitsu Laboratories Ltd.
We thank the editor and referees for very helpful comments which have greatly
improved the manuscript.


\end{document}